\documentclass[10pt]{article}
\usepackage[utf8]{inputenc}
\usepackage[T1]{fontenc}
\usepackage{amsmath}
\usepackage{amsfonts}
\usepackage{amsfonts}

\usepackage{xcolor}
\usepackage{amssymb}
\usepackage[version=4]{mhchem}
\usepackage{stmaryrd}
\usepackage{hyperref}
\hypersetup{colorlinks=true, linkcolor=blue, filecolor=magenta, urlcolor=cyan,}
\usepackage{bbold}

\usepackage{amsthm}

\newtheorem{dfn}{Definition}[section]
\newtheorem{lem}[dfn]{Lemma}
\newtheorem{thm}[dfn]{Theorem}
\newtheorem{cor}[dfn]{Corollary}

\newtheorem{prop}[dfn]{Proposition}

\newtheorem{remark}[dfn]{Remark}

\newtheorem{assumption}{Assumption}
\newcommand{\R}{\mathbb{R}}

\title{Third-Order Dynamical Systems for Generalized Inverse Mixed Variational Inequality Problems}

\author{Tran Van Nam ${ }^{1}$
\\
Email: namtv@hcmute.edu.vn}
\date{}

\let\svthefootnote\thefootnote
\newcommand\blfootnotetext[1]{%
  \let\thefootnote\relax\footnote{#1}%
  \addtocounter{footnote}{-1}%
  \let\thefootnote\svthefootnote%
}

\let\svfootnotetext\footnotetext
\renewcommand\footnotetext[2][?]{%
  \if\relax#1\relax%
    \ifnum\value{footnote}=0\blfootnotetext{#2}\else\svfootnotetext{#2}\fi%
  \else%
    \if?#1\ifnum\value{footnote}=0\blfootnotetext{#2}\else\svfootnotetext{#2}\fi%
    \else\svfootnotetext[#1]{#2}\fi%
  \fi
}

\DeclareUnicodeCharacter{22C5}{\ifmmode\cdot\else{$\cdot$}\fi}

\begin{document}
\maketitle

\begin{abstract}
In this paper, we propose and analyze a third-order dynamical system for solving a generalized inverse mixed variational inequality problem in a Hilbert space $H$. We establish the existence and uniqueness of the trajectories generated by the system under suitable continuity assumptions, and prove their exponential convergence to the unique solution under strong monotonicity and Lipschitz continuity conditions. Furthermore, we derive an explicit discretization of the proposed dynamical system, leading to a forward–backward algorithm with double inertial effects. We then establish the linear convergence of the generated iterates to the unique solution.
\end{abstract}



\section{Introduction}
\subsection{Background on Inverse Variational Inequalities and Dynamical Systems}
Let $H$ be a real Hilbert space equipped with the inner product $\langle \cdot, \cdot \rangle$ and the associated norm $\|\cdot\|$. Let $\Omega  \subset H$ be a nonempty closed convex set. Given a continuous mapping $\psi: H \to H$, the classical \textbf{variational inequality problem (VIP)} consists in determining $w^* \in \Omega$ such that
\begin{equation*}\label{pt1}
	\langle \psi(w^*), w - w^* \rangle \ge 0, \quad \forall w \in \Omega.
\end{equation*}
Variational inequalities provide a unifying mathematical framework for a wide spectrum of models arising in nonlinear analysis, economics, and engineering sciences \cite{10, 35}. Their theoretical and algorithmic developments have attracted considerable attention, particularly through projection-based iterative schemes \cite{5, 7, 34}.

In various practical applications, including transportation networks and market equilibrium models, the operator $\psi$ may not be directly accessible. Instead, the model is naturally formulated via an inverse mapping, leading to the \textbf{inverse variational inequality problem (IVIP)}. Let $F : H \to H$ be a continuous single-valued operator. The IVIP associated with $F$ and $\Omega$ aims to find $w^* \in H$ such that
\begin{equation*}\label{pt4}
	F(w^*) \in \Omega \quad \text{and} \quad \langle v - F(w^*), w^* \rangle \ge 0, \quad \forall v \in \Omega. 
\end{equation*}
Recent years have witnessed growing interest in inverse variational models \cite{2,11,17,18}. Nevertheless, many contemporary applications demand additional flexibility to incorporate nonsmooth terms or regularization effects. This requirement motivates the study of mixed-type formulations.

In this work, we consider a more general model, referred to as the \textbf{generalized inverse mixed variational inequality (GIMVI)}. The problem consists in finding $w^* \in H$ such that $F(w^*) \in \Omega$ and
\begin{equation}\label{gimvip}
	\langle g(w^*),\, v - F(w^*) \rangle
	+ h(v) - h\big(F(w^*)\big) \ge 0,
	\quad \forall v \in \Omega, 
\end{equation}
where $F, g : H \to H$ are continuous operators. This formulation strictly generalizes several existing inverse and mixed variational inequality models and offers a flexible framework for complex equilibrium problems \cite{22,JIT,Nam}.

Several well-known models are recovered as particular instances of \eqref{gimvip}. When $g$ is the identity operator, the problem reduces to an inverse mixed variational inequality. If additionally $h \equiv 0$, one recovers the classical inverse variational inequality problem. Conversely, if $F$ is the identity operator, \eqref{gimvip} becomes a mixed variational inequality, and further reduces to the classical VIP when $h \equiv 0$.

Parallel to these developments, an active line of research has explored the interplay between optimization theory and dynamical systems. Continuous-time models based on ordinary differential equations (ODEs) have proven to be powerful tools for analyzing and designing numerical algorithms for optimization \cite{2v,3v,9v,14v}, variational inequalities \cite{18v,27v,34v,43v}, monotone inclusions \cite{1v,5v,6v}, fixed point problems \cite{15v,17v}, equilibrium problems \cite{20v,36v,42v,45v}, and inverse quasi-variational inequalities \cite{VuongThanh, Vuong2026}. The ODE perspective not only clarifies the mechanism behind acceleration schemes such as Nesterov’s method, but also guides the construction of new algorithms with improved convergence behavior (see, e.g., \cite{9v,10v,18v,40v}).

\subsection{Some Historical Aspects}

Among inertial-type methods, the Heavy Ball with Friction system introduced by Polyak \cite{37v} plays a fundamental role. It is described by
\begin{equation}\label{pt1.4} 
	w^{(2)}(t)+\gamma w^{(1)}(t)+\nabla f(w(t))=0,
\end{equation}
and was designed to accelerate gradient-based minimization of a convex differentiable function $f$. When $\nabla f$ is Lipschitz continuous, it is co-coercive (see Section \ref{sec.prelim}). Extensions of \eqref{pt1.4} to constrained settings and co-coercive operators were developed in \cite{4v}. 

Boţ and Csetnek \cite{15v} later investigated second-order systems with time-dependent damping,
\[
w^{(2)}(t)+\gamma(t) w^{(1)}(t)+\lambda(t) B(w(t))=0,
\]
for locating zeros of a co-coercive operator $B$, and applied their results to forward-backward dynamics of the form
\begin{equation}\label{pt1.5} 
	w^{(2)}(t)+\gamma(t) w^{(1)}(t)+\lambda(t)\left[ w(t)- \mathcal{J}_{A} (w(t)-\gamma B(w(t)))\right]=0,
\end{equation}
where $A$ is maximal monotone and $\mathcal{J}_{A}=(I+A)^{-1}$ denotes its resolvent. For strongly monotone operators, exponential convergence of \eqref{pt1.5} was established in \cite{16v}, with rate $O(e^{-t})$ under suitable parameter conditions. Further developments for merely monotone operators were proposed in \cite{19v}.

Continuous-time approaches have also been applied to GIMVI-type problems \cite{1,19}. Most existing contributions focus on first-order systems such as
\begin{equation*}\label{pt11}
	w^{(1)}(t) = \rho \big[ P_K^{\gamma h}\big(F(w(t)) - \gamma g(w(t))\big) - F(w(t)) \big].
\end{equation*}
While first-order dynamics are well understood, their convergence speed may be limited due to the absence of inertial effects. Motivated by acceleration principles \cite{9,37}, second-order systems incorporating $\ddot{w}(t)$ have been proposed to enhance convergence and stability \cite{26,30}. For the GIMVI problem, a second-order time-varying system was studied in \cite{Nam2026}:
\begin{equation*}\label{pt12}
	\left\{
	\begin{aligned}
		w^{(2)}(t) &+ \kappa(t)w^{(1)}(t) + \rho(t)\big[ F(w(t)) - P_K^{\gamma h}\big(F(w(t)) - \gamma g(w(t))\big) \big] = 0, \\
		w(0) &= w_0, \quad \dot{w}(0) = v_0.
	\end{aligned}
	\right.
\end{equation*}

Higher-order dynamics have recently emerged as a promising direction. Attouch, Chbani and Riahi \cite{7v,8v} introduced third-order systems for convex minimization, including
\begin{equation*}\label{pt1.6} 
	w^{(3)}(t)+\frac{a}{t} w^{(2)}(t)+\frac{2 a-6}{t^{2}} w^{(1)}(t)+\nabla f\left(w(t)+t w^{(1)}(t)\right)=0.
\end{equation*}
Using temporal scaling and Lyapunov techniques, they obtained convergence rates of order $\frac{1}{t^{3}}$ for function values, and exponential convergence under strong convexity. An improved variant (TOGES-V) in \cite{8v} achieved $\mathcal{O}(1/t^{3})$ for $f(x(t))-\inf_{H}f$. 

For generalized monotone inclusion problems, third-order dynamics were investigated in \cite{Hai-Vuong}:
\begin{equation*}
	w^{(3)}(t)+\alpha_{2} w^{(2)}(t)+\alpha_{1} w^{(1)}(t)+\alpha_{0}\left[w(t)
	- \mathcal{J}_{A} (w(t)-\gamma B(w(t)))\right]=0,
\end{equation*}
where existence, uniqueness, and exponential convergence were established.

\subsection{Our Contributions}

Motivated by the above developments, we introduce for the first time a third-order dynamical system tailored to the generalized inverse mixed variational inequality problem \eqref{gimvip}. The analysis is carried out in both continuous and discrete frameworks.

Our results show that, under appropriate assumptions, the trajectories converge exponentially with rate $O\left(e^{-\varepsilon t}\right)$ for some $\varepsilon>1$ (in particular $\varepsilon=2$), which improves upon the classical second-order rate obtained in \cite{16v}. The explicit time discretization leads to a forward-backward algorithm incorporating double inertial effects and allowing a broader stepsize range.

The convergence analysis relies entirely on Lyapunov energy methods, in contrast to the temporal scaling techniques used in \cite{7v,8v}. The main contributions of this paper can be summarized as follows:

\begin{itemize}
	\item We propose a third-order dynamical system for the generalized inverse mixed variational inequality problem.
	\item We prove existence and uniqueness of its trajectories.
	\item We establish exponential convergence of the continuous-time trajectories to the unique solution.
	\item We derive an explicit discretization and prove linear convergence of the associated forward-backward algorithm with double inertial effects.
\end{itemize}

The remainder of the paper is organized as follows. Section \ref{sec.prelim} collects preliminary notions and auxiliary results. Section \ref{sec3} introduces the proposed third-order dynamical system and proves its exponential convergence to the solution of \eqref{gimvip}. In Section \ref{sec4}, we study its explicit discretization and analyze the corresponding forward-backward algorithm.

\section{Definitions and Preliminaries}\label{sec.prelim}
\subsection{Background on Monotone Operators and Generalized Projection Operators}
In this section, we present several definitions and auxiliary results that will be needed in the subsequent analysis.

Let $H$ be a real Hilbert space with inner product $\langle \cdot, \cdot \rangle$ and associated norm $\|\cdot\|$. 
We also consider the Cartesian product space $H^3 := H \times H \times H$, endowed with the inner product and norm introduced in \cite{Hai-Vuong}. For any $(x,y,z), (x^*, y^*, z^*) \in H^3$, define
\[
\langle (x,y,z), (x^*, y^*, z^*) \rangle
:= \langle x, x^* \rangle + \langle y, y^* \rangle + \langle z, z^* \rangle,
\]
and
\[
\|(x,y,z)\| := \sqrt{\|x\|^2 + \|y\|^2 + \|z\|^2}.
\]

We next recall several standard concepts concerning operator properties.

\begin{dfn}\cite{36}
	Let $\Omega$ be a nonempty subset of $H$. An operator $F : H \to H$ is said to be
	\begin{itemize}
		\item[(i)] \emph{monotone} on $\Omega$ if
		\[
		\langle F(w) - F(v), w - v \rangle \ge 0,
		\quad \forall\, w, v \in \Omega;
		\]
		\item[(ii)] \emph{strongly monotone} on $\Omega$ with modulus $\lambda > 0$ if
		\[
		\langle F(w) - F(v), w - v \rangle \ge \lambda \|w - v\|^2,
		\quad \forall\, w, v \in \Omega.
		\]
	\end{itemize}
\end{dfn}

\begin{dfn}
	Let $\Omega$ be a nonempty closed convex subset of $H$, and let $F, g : H \to H$ be two operators. 
	The pair $(F,g)$ is called \emph{$\zeta$-strongly coupled monotone} on $\Omega$ if there exists $\zeta > 0$ such that
	\[
	\langle F(w) - F(v), g(w) - g(v) \rangle
	\ge \zeta \|w - v\|^2,
	\quad \forall\, w, v \in \Omega.
	\]
\end{dfn}

It is clear that when $g$ coincides with the identity operator, this notion reduces to the classical definition of strong monotonicity.

\begin{dfn}\cite{36}
	An operator $F : H \to H$ is said to be \emph{$\eta$-Lipschitz continuous} on $\Omega$ if there exists $\eta \ge 0$ such that
	\[
	\|F(w) - F(v)\| \le \eta \|w - v\|,
	\quad \forall\, w, v \in \Omega.
	\]
\end{dfn}

Assume that $\Omega$ contains at least two distinct points. Suppose that $F$ and $g$ are Lipschitz continuous on $\Omega$ with constants $\eta$ and $\beta$, respectively, and that $(F,g)$ is $\zeta$-strongly coupled monotone on $\Omega$. Then one necessarily has
\[
\zeta \le \eta\,\beta.
\]
Similarly, if $F$ is $\eta$-Lipschitz continuous on $H$ and $\lambda$-strongly monotone, then
\[
\lambda \le \eta.
\]
We will also make use of the generalized $h$-projection operator (see \cite{Wu}), defined by
\begin{equation*}\label{pt3}
	P_{\Omega}^{\gamma h}(w)
	:= \arg\min_{v \in \Omega}
	\left\{
	\gamma h(v) + \frac{1}{2}\|w - v\|^2
	\right\},
\end{equation*}
where $\Omega$ is a nonempty closed convex subset of $H$, 
$h: \Omega \to \mathbb{R} \cup \{+\infty\}$ is a proper, convex, and lower semicontinuous function, and $\gamma > 0$. In convex analysis, this operator is commonly referred to as the \emph{proximal operator}.

In the particular case where $h$ is the indicator function of $\Omega$, namely
\[
h(w) =
\begin{cases}
	0, & w \in \Omega, \\
	+\infty, & w \notin \Omega,
\end{cases}
\]
the operator $P_{\Omega}^{\gamma h}$ coincides with the standard metric projection onto $\Omega$.

The following lemma collects several fundamental properties of the generalized $h$-projection.

\begin{lem}\label{lm2.1}\cite{35}
	Let $\Omega$ be a nonempty closed convex subset of $H$. Then:
	\begin{enumerate}
		\item[(i)]
		\[
		\langle w - P_{\Omega}^{\gamma h}(w), P_{\Omega}^{\gamma h}(w) - v \rangle \ge 0,
		\quad \forall\, w \in H,\ v \in \Omega;
		\]
		\item[(ii)] $P_{\Omega}^{\gamma h}$ is nonexpansive, that is,
		\[
		\|P_{\Omega}^{\gamma h}(w) - P_{\Omega}^{\gamma h}(v)\|
		\le \|w - v\|,
		\quad \forall\, w, v \in H;
		\]
		\item[(iii)] For $w \in H$, $w^* = P_{\Omega}^{\gamma h}(w)$ if and only if
		\[
		\langle w^* - w, v - w^* \rangle
		+ \gamma h(v) - \gamma h(w^*) \ge 0,
		\quad \forall\, v \in \Omega.
		\]
	\end{enumerate}
\end{lem}

Throughout the paper, we repeatedly use the Cauchy--Schwarz inequality:
for all $w, v \in H$,
\[
\langle w, v \rangle \le \|w\|\,\|v\|.
\]

\subsection{Absolutely Continuous Functions}
\begin{dfn}
	  A function $b: \mathbb{R}_{\geq 0} \rightarrow H$ is called locally absolutely continuous if it is absolutely continuous on every compact interval, which means that for each interval $[ t_{0}, t_{1}]$ there exists an integrable function $z:\left[t_{0}, t_{1}\right) \rightarrow H$ such that 	 
	 $$
	 b(t)=b\left(t_{0}\right)+\int_{t_{0}}^{t} z(s) d s \quad \forall t \in\left[t_{0}, t_{1}\right]
	 $$
	 
\end{dfn}
\begin{remark}
	 If $h: \mathbb{R}_{\geq 0} \rightarrow H$ is a locally absolutely continuous function, then it is differentiable almost everywhere and its derivative agrees with its distributional derivative almost everywhere.
\end{remark} 

\begin{prop}\cite{Hai-Vuong} \label{md2.8} 
	 For $s, a \geq 0$ and $n \in \mathbb{Z}_{\geq 1}$, it holds
	 
	 $$
	 \begin{aligned}
	 	\int_{a}^{s} e^{t \nu } z^{(n)}(t) d t= & e^{s \nu}\left(\sum_{j=0}^{n-1}(-\nu)^{n-1-j} z^{(j)}(s)\right)+(-\nu)^{n} \int_{a}^{s} e^{t \nu} z(t) d t \\
	 	& -e^{a \nu}\left(\sum_{j=0}^{n-1}(-\nu)^{n-1-j} z^{(j)}(a)\right) .
	 \end{aligned}
	 $$
\end{prop} 

\subsection{A Third Order Dynamical System}
In this paper, we propose the following dynamical system for Problem \eqref{gimvip}. 
\begin{equation}\label{pt2.1} 
w^{(3)}(t)+a_{2} w^{(2)}(t)+a_{1} w^{(1)}(t)+a_{0}\left[F(w(t))
- P_{\Omega}^{\gamma h}\!\left(
F(w(t)) - \gamma g(w(t))
\right)\right]=0 
\end{equation}
where $a_{2}, a_{1}, a_{0}, \gamma>0$ and $w^{(j)}\left(t_{0}\right)=v_{j}, j \in\{0,1,2\}$.\\
The solution of dynamical system \eqref{pt2.1} is understood in the following sense.\\
\begin{dfn}
	A function $w(\cdot)$ is called a strong global solution of equation \eqref{pt2.1} if it holds:	
	\begin{enumerate}
		\item For every $j \in\{0,1,2,3\}, w^{(j)}:\left[t_{0},+\infty\right) \rightarrow H$ is locally absolutely continuous; in other words, absolutely continuous on each interval $[d_1, d_2]$ for $d_2>d_1>t_{0}$.
		\item $w^{(3)}(t)+a_{2} w^{(2)}(t)+a_{1} w^{(1)}(t)+a_{0}\left[F(w(t))
		- P_{\Omega}^{\gamma h}\!\left(
		F(w(t)) - \gamma g(w(t))
		\right)\right]=0$ for almost every $t \geq t_{0}$.
		\item $w^{(j)}\left(t_{0}\right)=v_{j}, j \in\{0,1,2\}$.
	\end{enumerate}
\end{dfn}
For the sake of simplicity, we let 
$$\Psi(w(t)) =F(w(t))
- P_{\Omega}^{\gamma h}\!\left(
F(w(t)) - \gamma g(w(t))
\right).$$ 
We first derive several estimates for the operator $\Psi$, which will be useful in the sequel.
\begin{lem}\cite{Nam2026}
	Let $\Omega$ be a nonempty convex subset of $H$. 
	Assume that $F: H \to H$ is $\eta$-Lipschitz continuous and $\lambda$-strongly monotone on $\Omega$, 
	$h: \Omega \to \mathbb{R}\cup\{+\infty\}$ is proper, convex and lower semicontinuous, 
	and $g: H \to H$ is $\beta$-Lipschitz continuous. 
	Let $w^*$ be the unique solution of the GIMVIP \eqref{gimvip}, and suppose that $(F,g)$ is $\zeta$-strongly coupled monotone. 
	For any $\gamma>0$ and $w \in H$, define
	\begin{equation}\label{dna1}
		c_{1}:=\frac{c}{(2 \eta+\gamma\beta)^{2}},
	\end{equation}
	where
	\begin{equation}\label{dna}
		c:=\left(\lambda+\gamma \zeta-\frac{\eta^{2}}{2}-\frac{\gamma^{2}\beta}{2}-\frac{1}{2}\right).
	\end{equation}
	If $c>0$, then
	\begin{equation}\label{pt23}
		\left\langle w- w^* ,  A(w)\right\rangle \ge c_{1}\| A(w)\|^{2},
	\end{equation}
	\begin{equation}\label{pt23n}
		c\|w- w^*\|\le \|A(w)\|,
	\end{equation}
	and
	\begin{equation*}\label{pt24}
		\left\langle -A(w), w- w^* \right\rangle \le -c\left\|w- w^* \right\|^{2}.
	\end{equation*}
\end{lem}

\begin{prop}  Equation \eqref{pt2.1} can be reformulated as the first-order system $v^{(1)}(t)=\Phi(v(t))$, where the mapping $\Phi: H^3 \rightarrow H^3$ is given by
	$$
	\Phi\left(v_{1}, v_{2}, v_{3}\right)=\left(v_{2}, v_{3},-a_{1} v_{2}-a_{2} v_{3}-a_{0}\left[F(v_{1})- P_{\Omega}^{\gamma h}\!\left(
	F(v_1(t)) - \gamma g(v_1(t))
	\right)\right]\right), 
	$$	
for all $\left(v_{1}, v_{2}, v_{3}\right) \in H^3.	$
\end{prop}
\begin{proof}
	By setting 	
	$$
	\left(v_{1}(t), v_{2}(t), v_{3}(t)\right)=\left(w(t), w^{(1)}(t), w^{(2)}(t)\right)$$
	we obtain the desired conclusion. 	
\end{proof} 
\begin{thm} (Existence and uniqueness of a solution) Consider dynamcial system \eqref{pt2.1}, where $a_{0}, a_{1}, a_{2}, \gamma>0$ and let $F: H\longrightarrow H$ be an $\eta$-Lipschitz continuous operator and $g~:H \to H$ be a $\beta$-Lipschitz continuous operator, $h: \Omega\longrightarrow \R$ be a proper, convex, l.s.c function.. Then for each $v_{0}, v_{1}, v_{2} \in H$ there exists a unique strong global solution of \eqref{pt2.1}.
\end{thm} 
\begin{proof} We will prove that the operator $\Phi$ is Lipschitz. Indeed, put $w=\left(w_{1}, w_{2}, w_{3}\right), v= \left(v_{1}, v_{2}, v_{3}\right) \in H^3$, one has 
\begin{align*}\label{pt16} 
	\left\|\Psi w_1-\Psi w_2\right\| & =\left\|  F(w)-P_{\Omega}^{\gamma h}\left( F(w_1)-\gamma g(w_1)\right)- F\left(w_2\right)+P_{\Omega}^{\gamma h}\left( F(w_2)-\gamma g( w_2)\right) \right\| \notag \\
	& \leq \left\| F(w_1)- F(w_2) \right\|+ \left\|  P_{\Omega}^{\gamma h}\left( F(w_1)-\gamma g(w_1)\right)-P_{\Omega}^{\gamma h}\left( F(w_2)-\gamma g( w_2)\right) \right\| \notag \\
	& \leq \eta\|w_1-w_2\|+\|  F(w_1)-\gamma g(w_1)- F(w_2)+\gamma g( w_2)) \| \notag \\
	& \leq \eta\|w_1-w_2\|+\| F(w_1)- F(w_2)\|+\gamma\|g(w_1)-g\left(w_2\right)\| \notag \\
	& \leq \eta\|w_1-w_2\|+\eta\|w_1-w_2\|+\gamma\beta\|w_1-w_2\| \notag \\
	& \leq(2 \eta+\gamma\beta)\|w_1-w_2\|, \forall w_1, w_2\in H. 
\end{align*}	
Thus, $\Psi$ is $L_\Psi$-Lipschitz continuous with modulus $L_\Psi=2 \eta+\gamma\beta>0$, and so 
$$
\begin{aligned}
\| \Phi & (w)-\Phi(v)\left\|_{H^3}^{2} \leq\right\| w_{2}-v_{2}\left\|^{2}+\right\| w_{3}-v_{3} \|^{2} \\
& +\left(a_{1}^{2}+a_{2}^{2}+2 a_{0}^{2}\right)\left[\left\|w_{2}-v_{2}\right\|^{2}+\left\|w_{3}-v_{3}\right\|^{2}+\left\|w_{1}-v_{1}\right\|^{2}\right. \\
& \left.+\left\|\Psi(w_1)-\Psi(v_1)\right\|^{2}\right] \\
\leq & \left\|w_{2}-v_{2}\right\|^{2}+\left\|w_{3}-v_{3}\right\|^{2} \\
& +\left(a_{1}^{2}+a_{2}^{2}+2 a_{0}^{2}\right)\left[1+L_\Psi\right]\|w-v\|_{H^3}^{2} \\
\leq & {\left[1+\left(a_{1}^{2}+a_{2}^{2}+2 a_{0}^{2}\right)\left(1+L_\Psi\right)\right] \cdot\|w-v\|_{H^3}^{2} }
\end{aligned}
$$

By invoking the Cauchy-Picard theorem [\cite{28v}, Proposition 6.2.1], we obtain the desired conclusion.
\end{proof} 

\subsection{Difference Operators}
In this section, we introduce a discrete counterpart of the dynamical system \eqref{pt2.1}. For this purpose, we recall the definition of the forward difference operator together with several properties that will be used in the convergence analysis. Let $z: \mathbb{Z} \to H$ and $p \in \mathbb{Z}_{\ge 1}$. We denote

$$
w^{\Delta^{(p+1)}} := \left(w^{\Delta^{(p)}}\right)^{\Delta},  \text { where } w^{\Delta}(k) := w(k+1)-w(k).
$$

\begin{remark}\label{rm2.12} Let $x, y, z: \mathbb{Z} \rightarrow H$ and $\theta \in \mathbb{R}$. It can be proven that 	
	$$
	\langle y, z\rangle^{\Delta}(k)=\left\langle y^{\Delta}(k), z^{\Delta}(k)\right\rangle+\left\langle y^{\Delta}(k), z(k)\right\rangle+\left\langle y(k), z^{\Delta}(k)\right\rangle
	$$
By applying $y(k)=\theta^k$, the inner product $\langle y, z\rangle =yz$ and by writing 	$$(yz)^{\Delta}(k):=\big(\theta^kz\big)^{\Delta} (k)$$ we derive the following
	$$
	\begin{aligned}
		& \theta^{k+1} z^{\Delta}(k)=\left(\theta^{k} z\right)^{\Delta}(k)+(1-\theta) \theta^{k} z(k). 	
	\end{aligned}
	$$
Also, we have 
$$	 \left(\|x\|^{2}\right)^{\Delta}(k)=\left\|x^{\Delta}(k)\right\|^{2}+2\left\langle x^{\Delta}(k), x(k)\right\rangle.$$	
\end{remark}
We now introduce the following difference equation as a discrete analogue of \eqref{pt2.1}:
\begin{equation}\label{pt2.2} 
w^{\Delta^{(3)}}(k)+a_{2} w^{\Delta^{(2)}}(k)+a_{1} w^{\Delta}(k)+a_{0}\left[F(w(k))-P_K^{\gamma h} (F(w(k)) -\gamma g(w(k))) \right]=0
\end{equation}
where $a_{2}, a_{1}, a_{0}, \gamma>0$.\\
\begin{prop}   An equivalent representation of Equation \eqref{pt2.2} is given by 
\begin{align}\label{pt2.3} 
w(k+3)= & \left(3-a_{2}\right) w(k+2)+\left(2 a_{2}-a_{1}-3\right) w(k+1) \notag\\
& +\left(a_{1}+1-a_{2}\right) w(k)-a_{0}\left[\Psi(w(k))\right]. 
\end{align}
\end{prop} 
\begin{proof}
	 The proof relies on the following observations	 
	 $$
	 \begin{aligned}
	 	& w^{\Delta^{(2)}}(k)=w(k+2)-2 w(k+1)+w(k) \\
	 	& w^{\Delta^{(3)}}(k)=w(k+3)-3 w(k+2)+3 w(k+1)-w(k).
	 \end{aligned}
	 $$
	 
\end{proof}
\begin{remark}
	We can express the numerical scheme \eqref{pt2.3} equivalently as:	
	\begin{align*}
		w(k+3)= & w(k+2)+\left(2-a_{2}\right)(w(k+2)-w(k+1)) \\
		& +\left(a_{2}-a_{1}-1\right)(w(k+1)-w(k)) \\
		& -a_{0}\left[\Psi(w(k))\right] 
	\end{align*}	
which can be interpreted as a forward–backward algorithm incorporating double momentum effects.
\end{remark}

\section{Continuous Time Dynamical System}\label{sec3} 
This section is devoted to establishing the exponential convergence of the dynamical system \eqref{pt2.1}, based on the following assumptions and notations.

\begin{assumption} 	\label{a3.1} 
		Let $\Omega$ be a nonempty, closed, and convex subset of $\mathcal{H}$. Suppose that $F: \mathcal{H} \longrightarrow \mathcal{H}$ is $\eta$-Lipschitz continuous and $\lambda$-strongly monotone, $g: \mathcal{H} \longrightarrow \mathcal{H}$ is $\beta$-Lipschitz continuous, and $h: \Omega \longrightarrow \mathbb{R} \cup \{+\infty\}$ is a proper, convex, and l.s.c. function. Assume that the GIMVIP \eqref{gimvip} admits a unique solution $w^*$. 
		
		Furthermore, we assume the following conditions hold:
		\begin{enumerate}
			\item[(1)] The pair $(F, g)$ is $\zeta$-strongly coupled monotone.
			\item[(2)] There exists a constant $\gamma > 0$ such that 
			\begin{equation}\label{dna}
				c := \lambda + \gamma \zeta - \frac{\eta^2}{2} - \frac{\gamma^2 \beta}{2} - \frac{1}{2} > 0.
			\end{equation}
		\end{enumerate}
	\end{assumption}

\subsection{Global Exponential Convergence}
First, we examine the dynamical system \eqref{pt2.1}, whose global convergence is governed by the following parameters 

\begin{align} \label{pt3.6} 
	\begin{cases}  
		C_{ 2 } := \dfrac { c_1 a _ { 1 } } { a _ { 0 } }  \\
		 C_ { 1 } := \dfrac { c_1 a _ { 2 } a _ { 1 } } { a _ { 0 } } - 3  \\
		 C_{ 0 } := \dfrac { c_1 a _ { 1 } ^ { 2 } } { a _ { 0 } } - 2 a _ { 2 } 
	\end{cases} & \quad \begin{cases}
		A_{1} := \dfrac{c_1 a_{2}}{a_{0}}, \\
		A_{0} := \dfrac{c_1}{a_{0}}\left(a_{2}^{2}-2 a_{1}\right).
	\end{cases} 
\end{align} 
We denote the functions
\begin{equation*}\label{pt3.7} 
\phi(t) :=\left\|w(t)-w^*\right\|^{2}, \quad \varphi_{k}(t) :=\left\|w^{(k)}(t)\right\|^{2}.
\end{equation*}
\begin{thm}\label{dl3.2} 
	 Suppose that the operators $h, g, F$  satisfy Assumption \ref{a3.1}.  Let $w^*$ be the unique solution of Problem \eqref{gimvip}. Denote the parameters as in \eqref{dna}, \eqref{dna1}, \eqref{pt3.6}. Assume that there exists $\varepsilon>0$ such that the following conditions hold

	\begin{align}
		& -\varepsilon^{3}+a_{2} \varepsilon^{2}-a_{1} \varepsilon+c_1c^2 a_{0} \geq 0  \label{pt3.8}\\
		& C_{2} \varepsilon^{2}-C_{1} \varepsilon+C_{0} \geq 0  \label{pt3.9}\\
		& 3 \varepsilon^{2}-2 a_{2} \varepsilon+a_{1} \geq 0  \label{pt3.10}\\
		& -2 C_{2} \varepsilon+C_{1} \geq 0  \label{pt3.11}\\
		& -A_{1} \varepsilon+A_{0} \geq 0  \label{pt3.12}\\
		& a_{2}>2. \varepsilon \label{pt3.13}
	\end{align}	
	Then, the trajectories $w(\cdot)$ generated by the dynamical system \eqref{pt2.1} converge exponentially to $w^*$. That is, there exist positive constants $\mu$ and $\theta$ such that 
	$$
	\left\|w(t)-w^*\right\| \leq \mu\left\|w\left(t_{0}\right)-w^*\right\| e^{-\theta t} \quad \forall t \geq t_{0}.
	$$	
\end{thm}

\begin{proof}
	
In the next arguments, we often use the identities:
$$
\begin{aligned}
& \varphi_{1}^{(1)}(t)=2\left\langle w^{(2)}(t), w^{(1)}(t)\right\rangle \\
& \varphi_{1}^{(2)}(t)=2\left\langle w^{(3)}(t), w^{(1)}(t)\right\rangle+2\left\|w^{(2)}(t)\right\|^{2}=2\left\langle w^{(3)}(t), w^{(1)}(t)\right\rangle+2 \varphi_{2}(t)
\end{aligned}
$$
Because 
$$
\begin{aligned}
& \phi^{(1)}(t)=2\left\langle w^{(1)}(t), w(t)-w^*\right\rangle \\
& \phi^{(2)}(t)=2\left\langle w^{(2)}(t), w(t)-w^*\right\rangle+2 \varphi_{1}(t) \\
& \phi^{(3)}(t)=2\left\langle w^{(3)}(t), w(t)-w^*\right\rangle+3 \varphi_{1}^{(1)}(t)
\end{aligned}
$$
it holds that 
\begin{align*}
& 2\left\langle w^{(3)}(t)+a_{2} w^{(2)}(t)+a_{1} w^{(1)}(t), w(t)-w^*\right\rangle \\
& \quad=\phi^{(3)}(t)+a_{2} \phi^{(2)}(t)+a_{1} \phi^{(1)}(t)-3 \varphi_{1}^{(1)}(t)-2 a_{2} \varphi_{1}(t). 
\end{align*}
Using \eqref{pt2.1} yields 
\begin{align*}
	\phi^{(3)}(t)+a_2\phi^{(2)}(t)+a_1 \phi^{(1)} (t)	=& 2\langle w^{(3)}(t)+a_2w^{(2)}(t)+a_1w^{(1)}(t), w-w^* \rangle\\ &\quad +2a_2\varphi_1(t)+3\varphi_1^{(1)}(t)\\
	=&-2a_0\langle \Psi(w), w-w^*\rangle +2a_2\varphi_1(t)+3\varphi_1^{(1)}(t). 
\end{align*}
Consequently, 
\begin{align*}
		\phi^{(3)}(t)+a_2\phi^{(2)}(t)+a_1 \phi^{(1)}(t)+2a_0\langle \Psi(w), w-w^*\rangle -2a_2\varphi_1(t)-3\varphi_1^{(1)}(t)=0.
\end{align*}
By \eqref{pt23} it holds that 
\begin{align*}
	\phi^{(3)}(t)+a_2\phi^{(2)}(t)+a_1 \phi^{(1)}(t)+2a_0c_1\|\Psi(w)\| -2a_2\varphi_1(t)-3\varphi_1^{(1)}(t)\leq 0.
\end{align*}
Using \eqref{pt2.1} we get 
\begin{align} \label{ptnam1} 
	 &	\phi^{(3)}(t)+a_2\phi^{(2)}(t)+a_1 \phi^{(1)}(t)+a_0c_1\|\Psi(w)\|\notag \\
	 	&+ \frac{c_1}{a_0} \|w^{(3)}+a_2w^{(2)}+a_1w^{(1)}\|^2 -2a_2\varphi_1(t)-3\varphi_1^{(1)}(t)\leq 0. 
\end{align}
Observe that 
\begin{align*}
& \left\|w^{(3)}(t)+a_{2} w^{(2)}(t)+a_{1} w^{(1)}(t)\right\|^{2} \notag  \\
& \quad=a_{1} \varphi_{1}^{(2)}(t)+a_{2} a_{1} \varphi_{1}^{(1)}(t)+a_{1}^{2} \varphi_{1}(t)+a_{2} \varphi_{2}^{(1)}(t)+\left(a_{2}^{2}-2 a_{1}\right) \varphi_{2}(t)+\varphi_{3}(t). 
\end{align*}
Substituting this into \eqref{ptnam1} and using \eqref{pt23n} we get
\begin{align*}
&	\phi^{(3)}(t)+a_2\phi^{(2)}(t)+a_1 \phi^{(1)}(t)+a_0c_1c^2 \phi(t) \notag\\
	&+ \frac{c_1}{a_0} \left[ a_1\varphi_1^{(2)}(t)+ a_2a_1\varphi_1^{(1)}(t)+a_1^{2}\varphi_1(t)+a_2\varphi_2^{(1)}(t)-(a_2-2a_1)\varphi_2(t)+\varphi_3(t)\right] \\ &-2a_2\varphi_1(t)-3\varphi_1^{(1)}(t)\leq 0. 
\end{align*}
or 
\begin{align} \label{ptnam2} 
	&	\phi^{(3)}(t)+a_2\phi^{(2)}(t)+a_1 \phi^{(1)}(t)+a_0c_1c^2 \phi(t) \notag\\
	&+ \frac{c_1a_1}{a_0}\varphi_1^{(2)}(t) +\left( \frac{c_1a_1a_2}{a_0}-3\right) \varphi_1^{(1)}(t) +\left(\frac{a_1^2c_1 }{a_0}-2a_1\right)\varphi_1(t) \notag \\
	&+\frac{a_2c_1}{a_0}\varphi_2^{(1)}(t)-\left(a_2^2-2a_1\right) \frac{c_1}{a_0}\varphi_2(t)+\frac{c_1}{a_0}\varphi_3(t) \leq 0.    
\end{align}
By the assumption that $c>0$ and so is $c_1$ and $C_{0}$. Hence, \eqref{ptnam2} and \eqref{pt3.6} imply that 
$$
\begin{aligned}
& \phi^{(3)}(t)+a_{2} \phi^{(2)}(t)+a_{1} \phi^{(1)}(t)+ a_{0}c_1c^2 \phi(t) \\
& \quad+C_{2} \varphi_{1}^{(2)}(t)+C_{1} \varphi_{1}^{(1)}(t)+C_{0} \varphi_{1}(t)+A_{1} \varphi_{2}^{(1)}(t)+A_{0} \varphi_{2}(t) \leq 0
\end{aligned}
$$
We first multiply by $e^{\varepsilon\left(t-t_{0}\right)}$ and subsequently make use of Proposition \ref{md2.8}, we get 
 $$
\begin{aligned}
& e^{\varepsilon\left(s-t_{0}\right)}\left[\left(\varepsilon^{2}-a_{2} \varepsilon+a_{1}\right) \phi(s)+\left(a_{2}-\varepsilon\right) \phi^{(1)}(s)+\phi^{(2)}(s)\right] \\
& \quad+\left(-\varepsilon^{3}+a_{2} \varepsilon^{2}-a_{1} \varepsilon+a_0c_1c^2\right) \int_{t_{0}}^{s} e^{\varepsilon\left(t-t_{0}\right)} \phi(t) d t \\
& \quad+e^{\varepsilon\left(s-t_{0}\right)}\left[\left(-C_{2} \varepsilon+C_{1}\right) \varphi_{1}(s)+C_{2} \varphi_{1}^{(1)}(s)\right] \\
& \quad+\left(C_{2} \varepsilon^{2}-C_{1} \varepsilon+C_{0}\right) \int_{t_{0}}^{s} e^{\varepsilon\left(t-t_{0}\right)} \varphi_{1}(t) d t \\
& \quad+A_{1} e^{\varepsilon\left(s-t_{0}\right)} \varphi_{2}(s)+\left(-A_{1} \varepsilon+A_{0}\right) \int_{t_{0}}^{s} e^{\varepsilon\left(t-t_{0}\right)} \varphi_{2}(t) d t \leq R_{1}
\end{aligned}
$$
where $R_{1}$ is a constant. Applying \eqref{pt3.8}, \eqref{pt3.9} \eqref{pt3.12} yields 
$$
\begin{aligned}
& e^{\varepsilon\left(s-t_{0}\right)}\left[\left(\varepsilon^{2}-a_{2} \varepsilon+a_{1}\right) \phi(s)+\left(a_{2}-\varepsilon\right) \phi^{(1)}(s)+\phi^{(2)}(s)\right] \\
& \quad+e^{\varepsilon\left(s-t_{0}\right)}\left[\left(-C_{2} \varepsilon+C_{1}\right) \varphi_{1}(s)+C_{2} \varphi_{1}^{(1)}(s)\right]+A_{1} e^{\varepsilon\left(s-t_{0}\right)} \varphi_{2}(s) \leq R_{1}
\end{aligned}
$$
Intergrating the above inequality with respect to the variable $s \in\left[t_{0} ; t\right]$ we deduce 
$$
\begin{aligned}
& e^{\varepsilon\left(t-t_{0}\right)} \phi^{(1)}(t)+\left(a_{2}-2 \varepsilon\right) e^{\varepsilon\left(t-t_{0}\right)} \phi(t)+\left(3 \varepsilon^{2}-2 a_{2} \varepsilon+a_{1}\right) \int_{t_{0}}^{t} e^{\varepsilon\left(s-t_{0}\right)} \phi(s) d s \\
& \quad+C_{2} e^{\varepsilon\left(t-t_{0}\right)} \varphi_{1}(t)+\left(-2 C_{2} \varepsilon+C_{1}\right) \int_{t_{0}}^{t} e^{\varepsilon\left(s-t_{0}\right)} \varphi_{1}(s) d s \leq R_{1} t+R_{2}
\end{aligned}
$$
where $R_{2}$ is a constant. By \eqref{pt3.10}, \eqref{pt3.11} and the fact that $c_1>0$ and hence $C_2>0$,   one has 
\begin{equation}\label{pt3.18} 
e^{\varepsilon\left(t-t_{0}\right)} \phi^{(1)}(t)+\left(a_{2}-2 \varepsilon\right) e^{\varepsilon\left(t-t_{0}\right)} \phi(t) \leq R_{1} t+R_{2}. 
\end{equation} 
Note that equation \eqref{pt3.18}  reduces to the following
$$
\begin{aligned}
& \phi(t) \leq e^{-\left(a_{2}-2 \varepsilon\right)\left(t-t_{0}\right)} R_{3} \\
& \quad+e^{-\left(a_{2}-2 \varepsilon\right)\left(t-t_{0}\right)} \int_{t_{0}}^{t} e^{\left(a_{2}-3 \varepsilon\right)\left(s-t_{0}\right)}\left(R_{1} s+R_{2}\right) d s
\end{aligned}
$$
here $R_{3}$ is a constant.
\begin{itemize}
  \item If $a_{2} \geq 3 \varepsilon$, then $e^{\left(a_{2}-3 \varepsilon\right)\left(s-t_{0}\right)} \leq e^{\left(a_{2}-3 \varepsilon\right)\left(t-t_{0}\right)}$, and so
\end{itemize}
\begin{equation}\label{pt3.19}  
\phi(t) \leq e^{-\left(a_{2}-2 \varepsilon\right)\left(t-t_{0}\right)} R_{3}+e^{-\varepsilon\left(t-t_{0}\right)} \int_{t_{0}}^{t}\left(R_{1} s+R_{2}\right) d s.
\end{equation}
\begin{itemize}
  \item If $2 \varepsilon<a_{2}<3 \varepsilon$, then $e^{\left(a_{2}-3 \varepsilon\right)\left(s-t_{0}\right)} \leq 1$, and so
\end{itemize} 
\begin{equation}\label{pt3.20} 
\phi(t) \leq e^{-\left(a_{2}-2 \varepsilon\right)\left(t-t_{0}\right)}\left(R_{3}+\int_{t_{0}}^{t}\left(R_{1} s+R_{2}\right) d s\right).
\end{equation} 
The arguments above show that $w(\cdot)$ converges exponentially to $w^*$.
\end{proof} 
\begin{remark}
It follows from \eqref{pt3.19} that $\phi(t)$ converges to zero with rate
\[
O\!\left(\left(r_1t^{2}+r_2t+r_3t\right)e^{-\varepsilon t}\right)
\]
for some positive constants $r_1, r_2, r_3$. On the other hand, estimate \eqref{pt3.20} provides the decay rate
\[
O\!\left(\left(r_1t^{2}+r_2t+r_3t\right)e^{-(a_{2}-2\varepsilon)t}\right).
\]
By choosing $\varepsilon$ and $a_{2}$ suitably, these convergence rates can be made faster than the rate $O(e^{-t})$ obtained for the second-order dynamical systems in~\cite{Nam2026}.

\end{remark}

\subsection{On the Selection of Parameter Values}
We now address the selection of the parameter $\varepsilon$. As indicated by \eqref{pt3.19}, a larger $\varepsilon$ corresponds to a more rapid convergence rate. While determining the optimal (maximal) value of $\varepsilon$ is computationally demanding due to its complex dependence on various parameters, this section focuses on identifying a sufficiently effective value. The following remark provides a selection strategy based on the system's coefficients 

\begin{remark} \label{rmnam1} Observe that if 
\begin{equation}\label{pt3.21} 
C_{0}, C_{1}, A_{0}>0, 
\end{equation} 
Conditions \eqref{pt3.8}--\eqref{pt3.13} hold when $\varepsilon \rightarrow 0^{+}$.
\end{remark}
In the following result, we provide a simpler characterization of assumption \eqref{pt3.21} expressed through the algebraic relations of the coefficients $a_{0}, a_{1}$, and $a_{2}$. 
\begin{cor} \label{hq3.5} 
Consider equation \eqref{pt2.1}, under Assumption \ref{a3.1}. Let $w^*$ be the unique solution of Problem \eqref{gimvip}.  Denote parameters as in \eqref{dna}, \eqref{dna1}, \eqref{pt3.6}. Then $w(\cdot)$ converges exponentially to $w^*$ provided that coefficients $a_{0}, a_{1}, a_{2}, \gamma$ satisfy the following conditions 
\begin{align}
& a_{1}<\frac{a_{2}^{2}}{2} \label{dk1hq1} \\
& a_{0}<c_1 \cdot \min \left\{\frac{a_{1} a_{2}}{3}, \frac{a_{1}^{2}}{2 a_{2}}\right\} \label{dk2hq1}.
\end{align}
\end{cor}
\begin{proof}
	Observe that under conditions \eqref{dk1hq1} and \eqref{dk2hq1}, it holds that $C_{0}>0, C_{1}>0, A_{0}>0.$ By Remark \ref{rmnam1}, there exists $\varepsilon>0$ sufficiently  small such that Conditions \eqref{pt3.8}--\eqref{pt3.13} are satisfied. Hence, the conclusion follows from Theorem \ref{dl3.2}. 
\end{proof}
Setting $\varepsilon=1$ in Theorem \ref{dl3.2} recovers the convergence rate of the second-order dynamical system reported in \cite{Nam2026}. 
\begin{thm}\label{dl3.6} 
	  Suppose that the operators $h, g, F$  satisfy Assumption \ref{a3.1}. Let $w^*$ be the unique solution of Problem \eqref{gimvip}. Denote the parameters as in \eqref{dna1}, \eqref{dna}, \eqref{pt3.6} and	 
	 \begin{equation}\label{pt3.24} 
	 	\delta := \frac{1}{c_1^2c^2}
	 \end{equation}

	 Then $w(\cdot)$ converges exponentially to $w^*$ provided that coefficients $a_{0}, a_{1}, a_{2}$, $\gamma$ satisfy

	 \begin{align}
	 	& a_{2}>\max \{3,3 \delta+2,4 \delta\}  \label{pt3.25}\\
	 	& \underline{\beta} := \max \left\{2 a_{2}-3, \delta\left(2 a_{2}-3\right)\right\}<a_{1}<\bar{\beta} := 0.5 a_{2}\left(a_{2}-1\right)  \label{pt3.26}\\
	 	& \underline{q} := \frac{a_{1}-a_{2}+1}{c_1c^2}<a_{0}<\bar{p} := c_1 \cdot \min \left\{\frac{a_{1}\left(a_{2}-2\right)}{3}, \frac{a_{1}\left(a_{1}-a_{2}+1\right)}{2 a_{2}-3}\right\} \label{pt3.27}
	 \end{align}
\end{thm} 
\begin{proof}	 
 First, we show that \eqref{pt3.25} ensures the validity of \eqref{pt3.26}; that is $\underline{\beta}<\bar{\beta}$. Indeed, it follows from \eqref{pt3.25}  that $a_{2}>3$ and so $2 a_{2}-3<\frac{a_{2}\left(a_{2}-1\right)}{2}$. Also from \eqref{pt3.25}, we have $a_{2}>4 \delta$ and then 
$$
\frac{a_{2}\left(a_{2}-1\right)}{2}>\frac{a_{2}\left(2 a_{2}-3\right)}{4}>\delta\left(2 a_{2}-3\right) .
$$ 
Next, we show that \eqref{pt3.25}--\eqref{pt3.26}  ensure the validity of \eqref{pt3.27}; that is $\underline{q}<\bar{p}$. It results from \eqref{pt3.25} that $a_{2}>3 \delta+2$ and so 
$$
\frac{a_{1}-a_{2}+1}{c_1c^2}<\frac{a_{1}}{c_1c^2}<c_1 \cdot \frac{a_{1}\left(a_{2}-2\right)}{3} .
$$ 
Meanwhile, by \eqref{pt3.26}, we have $a_{1}>\delta\left( 2 a_{2}-3 \right)$, which gives 
$$
\frac{a_{1}-a_{2}+1}{c_1c^2}<c_1 \cdot \frac{a_{1}\left(a_{1}-a_{2}+1\right)}{2 a_{2}-3}.
$$
Note also that  when $\varepsilon=1$ 
 Condtions \eqref{pt3.8}--\eqref{pt3.13} become   	
\begin{align}
	& -1+a_{2} -a_{1}+c_1c^2 a_{0} \geq 0  \label{pt3.8n}\\
	& C_{2} -C_{1} +C_{0} \geq 0  \label{pt3.9n}\\
	& 3 -2 a_{2} +a_{1} \geq 0  \label{pt3.10n}\\
	& -2 C_{2} +C_{1} \geq 0  \label{pt3.11n}\\
	& -A_{1} +A_{0} \geq 0  \label{pt3.12n}\\
	& a_{2}>2.  \label{pt3.13n}
\end{align}	 
It is straightforward to verify that \eqref{pt3.13n} follows from \eqref{pt3.25}, 
while \eqref{pt3.8n} and \eqref{pt3.9n} are consequences of \eqref{pt3.27}. 
More precisely, \eqref{pt3.9n} holds since
\[
a_0 < c_1 \cdot \frac{a_{1}\left(a_{1}-a_{2}+1\right)}{2 a_{2}-3}.
\]
Moreover, \eqref{pt3.10n} is derived from \eqref{pt3.26}. 
Condition \eqref{pt3.11n} follows from the inequality
\[
a_0 < c_1 \cdot \frac{a_{1}(a_{2}-2)}{3},
\]
which is ensured by \eqref{pt3.27}. 
Finally, \eqref{pt3.12n} is satisfied in view of the bound
\[
a_1 < \tfrac{1}{2} a_{2}(a_{2}-1),
\]
as stated in \eqref{pt3.26}.
 
 Thus, all conditions in Theorem \ref{dl3.2} hold. Therefore, Theorem~\ref{dl3.2} can be invoked and hence we obtain the desired conclusion. 
\end{proof}

\noindent Now let us examine Theorem \ref{dl3.2} when $\varepsilon=2$. In this case, we will obtain from \eqref{pt3.19}  that the convergence rate of $\phi(t)$ is
$$
O\left(\left(r_1 t^{2}+r_2 t+r_3\right) e^{-2 t}\right),
$$
which is faster than the rate $O\left(e^{-t}\right)$ obtained in \cite{Nam2026} for the second order dynamical system.

\begin{thm}	 
 Suppose that the operators $h, g, F$  satisfy Assumption \ref{a3.1}. Let $w^*$ be the unique solution of Problem \eqref{gimvip}. Denote the parameters as in \eqref{dna1}, \eqref{dna}, \eqref{pt3.6} and \eqref{pt3.24}.  Then $w(\cdot)$ converges exponentially to $w^*$ provided that coefficients $a_{0}, a_{1}, a_{2}, \gamma$ satisfy
\begin{align}
& a_{2}>\max \{8 \delta, 6,6 \delta+4\}  \label{pt3.28}\\
& \underline{\beta} := \max \left\{4\left(a_{2}-3\right), 4 \delta\left(a_{2}-3\right)\right\}<a_{1}<\bar{\beta} := \frac{1}{2} a_{2}\left(a_{2}-2\right)  \label{pt3.29}\\
& \underline{q} := \frac{2}{c_1c^2}\left(a_{1}-2 a_{2}+4\right)<a_{0}<\bar{p} := c_1 a_{1} \cdot \min \left\{\frac{a_{1}-2 a_{2}+4}{2\left(a_{2}-3\right)}, \frac{1}{3}\left(a_{2}-4\right)\right\} \label{pt3.30}
\end{align}

\end{thm}
\begin{proof}
	
	Arguing as in Theorem~\ref{dl3.6}, we first verify that $\underline{\beta}<\bar{\beta}$. Observe that
	\[
	\frac{1}{2}a_2(a_2-2)>\frac{1}{2}a_2(a_2-3)>4\delta(a_2-3),
	\]
	which ensures the desired inequality.
	
	We next establish $\underline{q}<\bar{p}$. From \eqref{pt3.29} we have
	\[
	a_1>4(a_2-3)>2(a_2-2),
	\]
	and hence
	\begin{equation}\label{pt3.31}
		a_1-2a_2+4>0.
	\end{equation}
	Still by \eqref{pt3.29}, it holds that $a_1>4\delta(a_2-3)$, which implies
	\begin{equation}\label{pt3.32}
		\frac{c_1a_1}{2(a_2-3)}>\frac{2}{c_1c^2}.
	\end{equation}
	Combining \eqref{pt3.31} and \eqref{pt3.32}, we deduce
	\begin{equation}\label{pt3.33}
		c_1a_1\cdot\frac{a_1-2a_2+4}{2(a_2-3)}
		>
		\frac{2}{c_1c^2}(a_1-2a_2+4).
	\end{equation}	
	On the other hand, the inequality
	\[
	a_1(a_2-6\delta-4)+12\delta(a_2-2)>0
	\]
	is equivalent to
	\begin{equation}\label{pt3.34}
		c_1a_1\cdot\frac{1}{3}(a_2-4)
		>
		\frac{2}{\delta}(a_1-2a_2+4).
	\end{equation}
	Relations \eqref{pt3.33} and \eqref{pt3.34} together yield $\underline{q}<\bar{p}$.
	
	We now turn to the verification of \eqref{pt3.8}--\eqref{pt3.13}. For $\varepsilon=2$, these conditions reduce to
	\begin{align}
		& -8+4a_2-2a_1+c_1c^2a_0 \ge 0, \label{pt3.8nn}\\
		& 4C_2-2C_1+C_0 \ge 0, \label{pt3.9nn}\\
		& 12-4a_2+a_1 \ge 0, \label{pt3.10nn}\\
		& -4C_2+C_1 \ge 0, \label{pt3.11nn}\\
		& -2A_1+A_0 \ge 0, \label{pt3.12nn}\\
		& a_2>4. \label{pt3.13nn}
	\end{align}
	
	Conditions \eqref{pt3.8nn} and \eqref{pt3.13nn} are direct consequences of \eqref{pt3.28}. 
	Furthermore, \eqref{pt3.10nn} follows from the estimate
	\[
	4(a_2-3)<a_1,
	\]
	given in \eqref{pt3.29}. 
	
	Next, \eqref{pt3.9nn} is guaranteed by
	\[
	a_0<c_1a_1\cdot\frac{a_1-2a_2+4}{2(a_2-3)},
	\]
	see \eqref{pt3.30}. 
	In a similar way, \eqref{pt3.11nn} follows from
	\[
	a_0<c_1a_1\cdot\frac{1}{3}(a_2-4),
	\]
	which also appears in \eqref{pt3.30}. 
	Finally, \eqref{pt3.12nn} is satisfied thanks to the bound
	\[
	a_1<\frac{1}{2}a_2(a_2-2),
	\]
	as stated in \eqref{pt3.29}.
	
	Consequently, under assumptions \eqref{pt3.28}--\eqref{pt3.30}, all requirements \eqref{pt3.8}--\eqref{pt3.13} of Theorem~\ref{dl3.2} hold for $\varepsilon=2$. The conclusion of Theorem~\ref{dl3.2} therefore applies, and $\overline{y}(\cdot)$ converges exponentially to $w^*$. 
	
	Moreover, according to \eqref{pt3.19}, the convergence rate satisfies
	\[
	O\bigl((r_1t^2+r_2t+r_3)e^{-2t}\bigr)
	\]
	for some constants $r_1,r_2,r_3$.
	
\end{proof}

\section{Discrete Time Dynamical System}\label{sec4} 
"This section is devoted to proving the linear convergence of the numerical scheme \eqref{pt2.3} for solving \eqref{gimvip}, subject to the following supplementary assumption. 
We denote the following parameters 
\begin{align}
	& \left\{\begin{array} { l } 
		{ B_ { 2 } := \frac { c_1 a _ { 1 } } { a _ { 0 } } - 3 } \\
		{ B_ { 1 } := \frac { c_1 a _ { 2 } a _ { 1 } } { a _ { 0 } } - 2 a _ { 2 } - 3 } \\
		{ B_ { 0 } := \frac { c_1 a _ { 1 } ^ { 2 } } { a _ { 0 } } - 2 a _ { 2 } - a _ { 1 } }
	\end{array} \quad \left\{\begin{array}{l}
		D_{1} := \frac{c_1}{a_{0}}\left(a_{2}-2 a_{1}\right)+3 \\
		D_{0} := \frac{c_1}{a_{0}}\left(a_{2}^{2}-2 a_{1}-a_{2} a_{1}\right)+a_{2}+3
	\end{array}\right.\right.  \label{pt4.4}\\
	& E_{0} := \frac{c_1}{a_{0}}\left(1-a_{2}+a_{1}\right)-1 \label{pt4.5}
\end{align}
and
\begin{equation*}
	x(k) :=\left\|w(k)-w^*\right\|^{2}, \quad y_{p} :=\left\|w^{\Delta^{(p)}}(k)\right\|^{2} .
\end{equation*}
\begin{assumption}\label{a4.1} 
	 The coefficients $a_{0}, a_{1}, a_{2}$ satisfy 	 
	 \begin{align}
	 	& \frac{c_1}{a_{0}}\left(1-a_{2}+a_{1}\right)>1,  \label{pt4.1}\\
	 	& \frac{c_1}{a_{0}}\left(2 a_{1}-a_{2}\right)<3,  \label{pt4.2}\\
	 	& \frac{c_1 a_{1}}{a_{0}}>3, \label{pt4.3}
	 \end{align}  	 
	 where $c_1$ is defined in \eqref{dna1}. 
\end{assumption} 

\begin{remark}
	 
Under Assumption \ref{a4.1}, the positivity of $E_{0}, D_{1}$, and $B_{2}$ is guaranteed. Furthermore, Assumption \ref{a3.1} implies that the stepsize $\gamma$ must be subject to an upper bound, specifically:
$$
\gamma < \bar{\gamma} = \frac{c + \sqrt{c^2 + \beta \left( 2\lambda - \eta^2 - 1 \right)}}{\beta}.
$$
\end{remark} 
\subsection{Global Linear Convergence}
\begin{thm} \label{dl4.2} 
Suppose that the operators $h, g, F$  satisfy Assumption \ref{a3.1}. Let $w^*$ be the unique solution of Problem \eqref{gimvip}. Denote the parameters as in \eqref{dna1}, \eqref{dna}, \eqref{pt4.4} and \eqref{pt4.5}.   Assume that there exists $\xi>0, \xi \neq 1$ such that the following conditions hold 
\begin{align}
& -\xi^{3}+a_{2} \xi^{2}-a_{1} \xi+c_1c^2 a_{0} \geq 0  \label{pt4.7}\\
& B_{2} \xi^{2}-B_{1} \xi+B_{0} \geq 0  \label{pt4.8}\\
& 3 \xi^{2}-2 a_{2} \xi+a_{1} \geq 0  \label{pt4.9}\\
& -2 B_{2} \xi+B_{1} \geq 0  \label{pt4.10}\\
& -D_{1} \xi+D_{0} \geq 0  \label{pt4.11}\\
& a_{2}>3 \xi .\label{pt4.12}
\end{align} 
Then $w(k)$ converges linearly to $w^*$, i.e. there exist $Q>0$ and $q \in(0,1)$ such that 
$$
\left\|w(k)-w^*\right\| \leq Q q^{k} \quad \forall k .
$$
\end{thm} 
\begin{proof} 	
Since 
$$
\begin{aligned}
& x^{\Delta}(k)=2\left\langle w^{\Delta}(k), w(k)-w^*\right\rangle+y_{1}(k) \\
& x^{\Delta^{(2)}}(k)=2\left\langle w^{\Delta^{(2)}}(k), w(k)-w^*\right\rangle+2 y_{1}^{\Delta}(k)+2 y_{1}(k)-y_{2}(k), \\
& x^{\Delta^{(3)}}(k)=2\left\langle w^{\Delta^{(3)}}(k), w(k)-w^*\right\rangle \\
& \quad+3 y_{1}^{\Delta^{(2)}}(k)+3 y_{1}^{\Delta}(k)-3 y_{2}^{\Delta}(k)-3 y_{2}(k)+y_{3}(k),
\end{aligned}
$$
it follows that 
\begin{equation} \label{pt4.13}
\begin{aligned} 
2 & \left\langle w^{\Delta^{(3)}}(k)+a_{2} w^{\Delta^{(2)}}(k)+a_{1} w^{\Delta}(k), w(k)-w^*\right\rangle \\
& =x^{\Delta^{(3)}}(k)+a_{2} x^{\Delta^{(2)}}(k)+a_{1} x^{\Delta}(k) \\
& -3 y_{1}^{\Delta^{(2)}}(k)-\left(2 a_{2}+3\right) y_{1}^{\Delta}(k)-\left(2 a_{2}+a_{1}\right) y_{1}(k) \\
& +3 y_{2}^{\Delta}(k)+\left(a_{2}+3\right) y_{2}(k)-y_{3}(k). 
\end{aligned}
\end{equation} 
Applying \eqref{pt2.1} to \eqref{pt4.13}  yields 
\begin{align*}
	x^{\Delta^{(3)}}(k)+a_2x^{\Delta^{(2)}}(k)+a_1 x^{\Delta}(k)	=& 2\langle w^{\Delta^{(3)}}(k)+a_2w^{\Delta^{(2)}}(k)+a_1w^{\Delta}(k), w-w^* \rangle \\
	& +3 y_{1}^{\Delta^{(2)}}(k)+\left(2 a_{2}+3\right) y_{1}^{\Delta}(k)+\left(2 a_{2}+a_{1}\right) y_{1}(k) \\
	& -3 y_{2}^{\Delta}(k)-\left(a_{2}+3\right) y_{2}(k)+y_{3}(k)\\
	=&-2a_0\langle \Psi(w(k)), w(k)-w^*\rangle \\ 
	& +3 y_{1}^{\Delta^{(2)}}(k)+\left(2 a_{2}+3\right) y_{1}^{\Delta}(k)+\left(2 a_{2}+a_{1}\right) y_{1}(k) \\
	& -3 y_{2}^{\Delta}(k)-\left(a_{2}+3\right) y_{2}(k)+y_{3}(k).
\end{align*}
Consequently, 
\begin{align*}
&	x^{\Delta^{(3)}}(k)+a_2x^{\Delta^{(2)}}(k)+a_1 x^{\Delta}(k)	+2a_0\langle \Psi(w(k)), w(k)-w^*\rangle  \\
&	-3 y_{1}^{\Delta^{(2)}}(k)-\left(2 a_{2}+3\right) y_{1}^{\Delta}(k)-\left(2 a_{2}+a_{1}\right) y_{1}(k) \\
	& +3 y_{2}^{\Delta}(k)+\left(a_{2}+3\right) y_{2}(k)-y_{3}(k)=0.
\end{align*}
By \eqref{pt23} it holds that 
\begin{align*}
	&	x^{\Delta^{(3)}}(k)+a_2x^{\Delta^{(2)}}(k)+a_1 x^{\Delta}(k)	+2c_1a_0 \|\Psi(w(k))\|^2 \\
	&	-3 y_{1}^{\Delta^{(2)}}(k)-\left(2 a_{2}+3\right) y_{1}^{\Delta}(k)-\left(2 a_{2}+a_{1}\right) y_{1}(k) \\
	& +3 y_{2}^{\Delta}(k)+\left(a_{2}+3\right) y_{2}(k)-y_{3}(k)\leq 0.
\end{align*}
Using \eqref{pt2.2} we get 
\begin{align*}
	&	x^{\Delta^{(3)}}(k)+a_2x^{\Delta^{(2)}}(k)+a_1 x^{\Delta}(k)	+c_1a_0 \|\Psi(w(k))\|^2 \\
	&+  \frac{c_1}{a_0}\|w^{\Delta^{(3)}}(k)+a_2w^{\Delta^{(2)}}(k)+a_1w^{\Delta}(k) \|^2\\
	&	-3 y_{1}^{\Delta^{(2)}}(k)-\left(2 a_{2}+3\right) y_{1}^{\Delta}(k)-\left(2 a_{2}+a_{1}\right) y_{1}(k) \\
	& +3 y_{2}^{\Delta}(k)+\left(a_{2}+3\right) y_{2}(k)-y_{3}(k)\leq 0.
\end{align*}
Observe that 
\begin{equation*} 
\begin{aligned}
	& \left\|w^{\Delta^{(3)}}(k)+a_{2} w^{\Delta^{(2)}}(k)+a_{1} w^{\Delta}(k)\right\|^{2} \\
	& \quad=a_{1} y_{1}^{\Delta^{(2)}}(k)+a_{2} a_{1} y_{1}^{\Delta}(k)+a_{1}^{2} y_{1}(k) \\
	& \quad+\left(a_{2}-2 a_{1}\right) y_{2}^{\Delta}(k)+\left(a_{2}^{2}-2 a_{1}-a_{2} a_{1}\right) y_{2}(k)+\left(1-a_{2}+a_{1}\right) y_{3}(k).
\end{aligned}
\end{equation*} 
Substituting this into \eqref{ptnam1} and using \eqref{pt23n} we get
\begin{align*}
	&	x^{\Delta^{(3)}}(k)+a_2x^{\Delta^{(2)}}(k)+a_1 x^{\Delta}(k)	+c_1a_0c^2 x(k) \\
	&+ \frac{c_1}{a_0} \Big[ a_{1} y_{1}^{\Delta^{(2)}}(k)+a_{2} a_{1} y_{1}^{\Delta}(k)+a_{1}^{2} y_{1}(k) \\
	& \quad+\left(a_{2}-2 a_{1}\right) y_{2}^{\Delta}(k)+\left(a_{2}^{2}-2 a_{1}-a_{2} a_{1}\right) y_{2}(k)\\
	&+\left(1-a_{2}+a_{1}\right) y_{3}(k) \Big] \\
	&	-3 y_{1}^{\Delta^{(2)}}(k)-\left(2 a_{2}+3\right) y_{1}^{\Delta}(k)-\left(2 a_{2}+a_{1}\right) y_{1}(k) \\
	& +3 y_{2}^{\Delta}(k)+\left(a_{2}+3\right) y_{2}(k)-y_{3}(k)\leq 0.
\end{align*}
Rearranging, we get 
\begin{align*}
	&	x^{\Delta^{(3)}}(k)+a_2x^{\Delta^{(2)}}(k)+a_1 x^{\Delta}(k)	+c_1a_0c^2 x(k) \\
	&+ \left(\frac{c_1}{a_0}a_{1} -3\right) y_{1}^{\Delta^{(2)}}(k)+ \left(  \frac{c_1}{a_0}a_{2} a_{1}-(2a_2+3)\right) y_{1}^{\Delta}(k)\\
	&+ \left(  \frac{c_1}{a_0}a_{1}^{2}-(2a_2+a_1)\right)  y_{1}(k) \\
	& \quad+\left( \frac{c_1}{a_0}\left(a_{2}-2 a_{1}\right)+3\right) y_{2}^{\Delta}(k)+ \left(  \frac{c_1}{a_0}\left(a_{2}^{2}-2 a_{1}-a_{2} a_{1}\right)+(a_2+3)\right) y_{2}(k)\\
	&+\left(  \frac{c_1}{a_0}\left(1-a_{2}+a_{1}\right)-1\right) y_{3}(k) \leq 0.
\end{align*}
By Notation \eqref{pt4.4} and \eqref{pt4.5} this inequality can be rewritten as follows: 
$$
\begin{aligned}
& x^{\Delta^{(3)}}(k)+a_{2} x^{\Delta^{(2)}}(k)+a_{1} x^{\Delta}(k)+c_1c^2 a_{0} x(k) \\
& \quad+B_{2} y_{1}^{\Delta^{(2)}}(k)+B_{1} y_{1}^{\Delta}(k)+B_{0} y_{1}(k)+D_{1} y_{2}^{\Delta}(k) +D_{0} y_{2}(k)+E_{0} y_{3}(k) \leq 0
\end{aligned}
$$
By the assumption \eqref{pt4.1}, $E_0>0$. Hence, the inequality above implies  that 
$$
\begin{aligned}
& x^{\Delta^{(3)}}(k)+a_{2} x^{\Delta^{(2)}}(k)+a_{1} x^{\Delta}(k)+c_1c^2 a_{0} x(k) \\
& \quad+B_{2} y_{1}^{\Delta^{(2)}}(k)+B_{1} y_{1}^{\Delta}(k)+B_{0} y_{1}(k)+D_{1} y_{2}^{\Delta}(k)+D_{0} y_{2}(k) \leq 0
\end{aligned}
$$
Setting 
$$
l := \frac{1}{1-\xi}
$$ 
Then $l>1$ and conditions \eqref{pt4.7}--\eqref{pt4.12} can be written as 
\begin{align}
& c_1c^2 a_{0} \varepsilon^{3}+a_{1} l^{2}(1-l)+\left(a_{2} l+1-l\right)(1-l)^{2} \geq 0,  \label{pt4.16}\\
& B_{0} l^{2}+B_{1}(1-l) l+B_{2}(1-l)^{2} \geq 0,  \label{pt4.17}\\
& a_{1} l^{2}+2 a_{2} l(1-l)+3(1-l)^{2} \geq 0,  \label{pt4.18}\\
& B_{1} l+2 B_{2}(1-l) \geq 0,  \label{pt4.19}\\
& D_{0} l+D_{1}(1-l) \geq 0,  \label{pt4.20}\\
& l a_{2}+3(1-l)>0 . \label{pt4.21}
\end{align}
Multiplying both sides by $l^{k+3}$ and subsequently invoking Remark \ref{rm2.12} leads to 
$$
\begin{aligned}
& \left(l^{k+2} x^{\Delta^{(2)}}\right)^{\Delta}(k)+\left(a_{2} l+1-l\right)\left(l^{k+1} x^{\Delta}\right)^{\Delta}(k) \\
& +\left[a_{1} l^{2}+\left(a_{2} l+1-l\right)(1-l)\right]\left(l^{k} x\right)^{\Delta}(k) \\
& +\underbrace{\left[c_1c^2 a_{0} l^{3}+a_{1} l^{2}(1-l)+\left(a_{2} l+1-l\right)(1-l)^{2}\right]}_{\geq 0 \text { (by }\eqref{pt4.16})} l^{k} x(k) \\
& +B_{2}\left(l^{k+2} y_{1}^{\Delta}\right)^{\Delta}(k)+\left[B_{1} l-B_{2}(l-1)\right]\left(l^{k+1} y_{1}\right)^{\Delta}(k) \\
& +\underbrace{\left[B_{0} l^{2}+B_{1}(1-l) l+B_{2}(1-l)^{2}\right]}_{\geq 0 \text { (by }\eqref{pt4.17})} l^{k+1} y_{1}(k) \\
& +D_{1}\left(l^{k+2} y_{2}\right)^{\Delta}(k)+\underbrace{\left[-D_{1}(l-1)+D_{0} l\right]}_{\geq 0 \text { (by }\eqref{pt4.20})} l^{k+2} y_{2}(k) \leq 0 .
\end{aligned}
$$
Let $n \in \mathbb{Z}_{\geq 1}$ and summing both sides from $k=0$ to $n-1$ yields 
$$
\begin{aligned}
& l^{n+2} x^{\Delta^{(2)}}(n)+\left(a_{2} l+1-l\right) l^{n+1} x^{\Delta}(n)+\left[a_{1} l^{2}+\left(a_{2} l+1-l\right)(1-l)\right] l^{n} x(n) \\
& \quad+B_{2} l^{n+2} y_{1}^{\Delta}(n)+\left[B_{1} l-B_{2}(l-1)\right] l^{n+1} y_{1}(n)+\underbrace{D_{1} l^{n+2} y_{2}(n)}_{\geq 0 \quad \text { (by \eqref{pt4.2}) }} \leq Q_{1}
\end{aligned}
$$
for some constant $Q_{1}$. Invoking Remark \ref{rm2.12} once again, we have,
$$
\begin{aligned}
& \left(l^{n+1} x^{\Delta}\right)^{\Delta}(n)+\left[a_{2} l+2(1-l)\right]\left(l^{n} x\right)^{\Delta}(n) \\
& \quad+\underbrace{\left[a_{1} l^{2}+2 a_{2} l(1-l)+3(1-l)^{2}\right]}_{\geq 0 \quad(\text { by }\eqref{pt4.18})} l^{n} x(n) \\
& \quad+B_{2}\left(l^{n+1} y_{1}\right)^{\Delta}(n)+\underbrace{\left[B_{1} l-2 B_{2}(l-1)\right]}_{\geq 0 \quad(\text { by }\eqref{pt4.19})} l^{n+1} y_{1}(n) \leq Q_{1}
\end{aligned}
$$
Let $m \in \mathbb{Z}_{\geq 2}$ and summing both sides over $n$ from $1$ to $m-1$ yields 
$$
l^{m+1} x^{\Delta}(m)+\left[a_{2} l+2(1-l)\right] l^{m} x(m)+\underbrace{B_{2} l^{m+1} y_{1}(m)}_{\geq 0 \quad(\text { by }\eqref{pt4.3})} \leq Q_{1} m+Q_{2}
$$
for some constant  $Q_{2}$. Invoking Remark \ref{rm2.12} once again, we obtain,
$$
\left(l^{m} x\right)^{\Delta}(m)+\underbrace{\left[a_{2} l+3(1-l)\right] l^{m} x(m)}_{\geq 0(\text { by }\eqref{pt4.21})} \leq Q_{1} m+Q_{2}
$$
which leads to the following result upon summing from $m=2$ to $k-1$
$$
l^{k} x(k) \leq Q_{1} k^{2}+Q_{2} k+Q_{3} \leq Q_{4} k^{2} .
$$
Here $k \in \mathbb{Z}_{\geq 3}$ and $Q_{3}, Q_{4}$ are some positive constants. Let $q$ such that $1<q<l$. We have
$$
x(k) \leq \frac{Q_{4} k^{2}}{l^{k}}=\left(\frac{q}{l}\right)^{k} \cdot \frac{Q_{4} k^{2}}{q^{k}} \leq Q_{5}\left(\frac{q}{l}\right)^{k},
$$
for some constant $Q_{5}$. The above inequality implies that the sequence $\{w(k)\}$ converges linearly to $w^*$.
\end{proof}
\subsection{On Seclection of Parameter Values}
We now turn to the selection of parameters that satisfy the assumptions required by Theorem \ref{dl4.2}. It is important to observe that if $B_{0}, B_{1}$, and $D_{0}$ fulfill the condition:
\begin{equation}\label{pt4.22}
	B_{0}, B_{1}, D_{0} > 0,
\end{equation}
then the constraints \eqref{pt4.16}--\eqref{pt4.21} are satisfied in the limit as $\xi \to 0^{+}$. 

The following result provides a more tractable characterization of \eqref{pt4.22} in terms of the algebraic coefficients $a_{0}, a_{1}$, and $a_{2}$.

\begin{cor}\label{hq4.3} 
	 Suppose that the operators $h, g, F$  satisfy Assumption \ref{a3.1}. Let $w^*$ be the unique solution of Problem \eqref{gimvip}. Denote the parameters as in \eqref{dna1}, \eqref{dna}, \eqref{pt4.4} and \eqref{pt4.5}.  Then $w(k)$ converges linearly to $w^*$ provided that coefficients $a_{0}, a_{1}, a_{2}, \gamma$ satisfy

	\begin{align*}
		& a_{2}<2 \\
		& \max \left\{0, a_{2}-1\right\}<a_{1}<\frac{a_{2}^{2}}{a_{2}+2} \\
		& a_{0}<c_1 \cdot \min \left\{\frac{a_{1}^{2}}{a_{1}+2 a_{2}}, 1-a_{2}+a_{1}\right\} 
	\end{align*}
\end{cor} 
\begin{proof}
 Since $a_{1}<\frac{a_{2}^{2}}{a_{2}+2}$, we have $D_{0}>a_{2}+3>0$. Also using $a_{1}<\frac{a_{2}^{2}}{a_{2}+2}$ and the fact that $a_{2}<2$, we get

\begin{equation}\label{pt4.26} 
a_{1}<\frac{a_{2}^{2}}{a_{2}+2}<\frac{a_{2}}{2}, 
\end{equation} 
which gives \eqref{pt4.2}. It follows from \eqref{pt4.26}  that $a_{1}<a_{2}$ and so 
$$
\frac{c_1 a_{1}}{3}>\frac{c_1 a_{1}^{2}}{a_{1}+2 a_{2}}>a_{0}
$$ 
The last inequality proves \eqref{pt4.3}. Thus, Assumption \ref{a4.1} holds. Note that 
$$
a_{1}<\frac{a_{2}^{2}}{a_{2}+2}<\frac{2 a_{2}^{2}}{a_{2}+3},
$$ 
which gives 
$$
\frac{c_1 a_{1} a_{2}}{2 a_{2}+3}>\frac{c_1 a_{1}^{2}}{a_{1}+2 a_{2}}>a_{0}
$$ 
and then $B_{1}>0$.
\end{proof}
\begin{remark} 	 

It is worth noting that there exist common parameter selections satisfying the conditions of both Corollary \ref{hq3.5} (in the limit $\varepsilon \to 0$) and Corollary \ref{hq4.3} (as $\xi \to 0$). For instance, one may verify that the following selection is sufficient:
\begin{equation*}
	\begin{aligned}
		& a_{2} < 1, \\
		& a_{1} < \frac{a_{2}^{2}}{a_{2}+2}, \\
		& a_{0} < c_{1} \cdot \min \left\{ \frac{1}{3} a_{1} a_{2}, \frac{a_{1}^{2}}{a_{1} + 2 a_{2}} \right\}.
	\end{aligned}
\end{equation*}
\end{remark}
\begin{remark}
Since the generalized inverse mixed variational inequality problem (GIMVIP) \eqref{gimvip} encompasses several variational problems as special cases, one can deduce specific results for corresponding problems by choosing appropriate operators or functions. 

Specifically, if $g$ is the identity operator, the GIMVIP \eqref{gimvip} reduces to an inverse mixed variational inequality problem. In this case, the dynamical system \eqref{pt2.1} simplifies to
\begin{equation}\label{pt2.1n} 
	w^{(3)}(t)+a_{2} w^{(2)}(t)+a_{1} w^{(1)}(t)+a_{0}\left[F(w(t)) - P_{\Omega}^{\gamma h}(F(w(t)) - \gamma w(t))\right]=0.
\end{equation}
Moreover, if $h \equiv 0$, then $P_{\Omega}^{\gamma h}$ reduces to the projection operator $P_{\Omega}$, and equation \eqref{pt2.1n} becomes
\begin{equation*}
	w^{(3)}(t)+a_{2} w^{(2)}(t)+a_{1} w^{(1)}(t)+a_{0}\left[F(w(t)) - P_{\Omega}(F(w(t)) - \gamma w(t))\right]=0.
\end{equation*}

When $F$ is the identity operator, the GIMVIP \eqref{gimvip} recovers a mixed variational inequality problem, and equation \eqref{pt2.1} collapses to
\begin{equation*}\label{pt2.1n3} 
	w^{(3)}(t)+a_{2} w^{(2)}(t)+a_{1} w^{(1)}(t)+a_{0}\left[w(t) - P_{\Omega}^{\gamma h}(w(t) - \gamma g(w(t)))\right]=0.
\end{equation*}
In addition, if $h \equiv 0$, the model further simplifies to a classical variational inequality problem:
\begin{equation*}\label{pt2.1n4} 
	w^{(3)}(t)+a_{2} w^{(2)}(t)+a_{1} w^{(1)}(t)+a_{0}\left[w(t) - P_{\Omega}(w(t) - \gamma g(w(t)))\right]=0,
\end{equation*}
which was studied in \cite{Hai-Vuong}.
\end{remark}

\end{document}